\documentclass{article}

\usepackage{amsmath,amsthm}
\usepackage{mathrsfs} 
\usepackage{amssymb}
\usepackage{amsfonts}
\usepackage{upgreek}
\usepackage{nicefrac}
\usepackage{geometry}
\usepackage{authblk}
\usepackage{graphicx}
\usepackage{hyperref}
\hypersetup{
  colorlinks = true,
  citecolor = blue,
  linkcolor = blue,
}
 \usepackage{color}

\usepackage{enumitem}
\usepackage{url}
\usepackage{graphicx} 
\usepackage{lmodern} 
\usepackage{xcolor}

\usepackage{ifthen}
\usepackage{xargs}

\usepackage[style=alphabetic,giveninits,doi=false,isbn=false,url=false,eprint=false,clearlang=true,sorting=nyt,maxbibnames=99,noerroretextools]{biblatex}
\DeclareSourcemap{
      \maps[datatype=bibtex]{
      \map[overwrite=false]{
       \step[fieldsource=note]
       \step[fieldset=addendum, origfieldval, final]
       \step[fieldset=note, null]
       }
   }
   }

\usepackage{aliascnt}
\usepackage{cleveref}
\let\etoolboxforlistloop\forlistloop 
\usepackage{autonum}
\let\forlistloop\etoolboxforlistloop 

\makeatletter

\newtheorem{theorem}{Theorem}[section]
\crefname{theorem}{theorem}{Theorems}
\Crefname{Theorem}{Theorem}{Theorems}

\newaliascnt{lemma}{theorem}
\newtheorem{lemma}[theorem]{Lemma}
\aliascntresetthe{lemma}
\crefname{lemma}{lemma}{lemmas}
\Crefname{Lemma}{Lemma}{Lemmas}

\newaliascnt{corollary}{theorem}
\newtheorem{corollary}[theorem]{Corollary}
\aliascntresetthe{corollary}
\crefname{corollary}{corollary}{corollaries}
\Crefname{Corollary}{Corollary}{Corollaries}

\newaliascnt{proposition}{theorem}
\newtheorem{proposition}[theorem]{Proposition}
\aliascntresetthe{proposition}
\crefname{proposition}{proposition}{propositions}
\Crefname{Proposition}{Proposition}{Propositions}

\newaliascnt{definition}{theorem}
\newtheorem{definition}[theorem]{Definition}
\aliascntresetthe{definition}
\crefname{definition}{definition}{definitions}
\Crefname{Definition}{Definition}{Definitions}

\newaliascnt{remark}{theorem}
\newtheorem{remark}[theorem]{Remark}
\aliascntresetthe{remark}
\crefname{remark}{remark}{remarks}
\Crefname{Remark}{Remark}{Remarks}

\crefname{example}{example}{examples}
\Crefname{Example}{Example}{Examples}

\newtheorem{assumption}{\textbf{A}\hspace{-3pt}}
\crefformat{assumption}{{\textbf{A}}#2#1#3}

\newtheorem{assumptionalt}{\textbf{A}\hspace{-3pt}}[assumption]
\crefformat{assumptionalt}{{\textbf{A}}#2#1#3}

\newenvironment{assumptionp}[1]{
  \renewcommand\theassumptionalt{#1}
  \assumptionalt
}{\endassumptionalt}

\crefformat{equation}{(#2#1#3)}
\renewcommand{\eqref}{\cref}

\definecolor{forestgreen}{rgb}{0.13, 0.55, 0.13}

\newcommand{\IND}{\mathbf{1}}

\newcommand{\cF}{\mathcal{F}}

\newcommand{\cP}{\mathcal{P}}

\newcommand{\cV}{\mathcal{V}}

\newcommand{\cL}{\mathcal{L}}

\newcommand{\bbR}{\mathbb{R}}

\newcommand{\bbP}{\mathbb{P}}
\newcommand{\bbE}{\mathbb{E}}

\newcommand{\De}{\mathrm{d}}

\def\msk{\mathsf{K}}

\def\msc{\mathsf{C}}

\newcommandx{\Vnorm}[2][1=V]{\| #2 \|_{#1}}
\newcommandx{\norm}[2][1=]{\ifthenelse{\equal{#1}{}}{\left\Vert #2 \right\Vert}{\left\Vert #2 \right\Vert^{#1}}}
\newcommandx{\normLigne}[2][1=]{\ifthenelse{\equal{#1}{}}{\Vert #2 \Vert}{\Vert #2\Vert^{#1}}}

\newcommand{\defEns}[1]{\left\lbrace #1 \right\rbrace }

\def\eqsp{\;}

\def\transpose{\top}
\DeclareMathOperator{\Id}{\mathrm{Id}}

\renewcommand{\epsilon}{\varepsilon}

\newcommand{\ip}[2]{\langle #1,#2\rangle}

\def\const{\mathrm{C}}

\newcommand{\TV}{\mathrm{TV}}

\title{A coupling approach to Lipschitz transport maps}
\author{Giovanni Conforti, Katharina Eichinger }

\begin{document}

\maketitle

\begin{abstract}
   In this note, we propose a probabilistic approach to bound the (dimension-free) Lipschitz constant of the Langevin flow map on $\bbR^d$ introduced by Kim and Milman \cite{kim2012GeneralizationCaffarellisContraction}. As example of application, we construct Lipschitz maps from a uniformly $\log$-concave probability measure to $\log$-Lipschitz perturbations as in  \cite{fathi2024TransportationLogLipschitzPerturbations}. Our proof is based on coupling techniques applied to the stochastic representation of the family of
   vector fields inducing the transport map. This method is robust enough to relax the uniform convexity to a weak asymptotic convexity condition and to remove the bound on the third derivative of the potential of the source measure.
\end{abstract}

\section{Introduction}
The goal of this note is to provide an alternative proof of the Lipschitz continuity of the transport map introduced by Kim and Milman \cite{kim2012GeneralizationCaffarellisContraction} with an explicit (dimension-free) bound on the Lipschitz constant. Given $\mu, \nu \in \cP(\bbR^d)$, a transport map from $\mu$ to $\nu$  is a map $T: \bbR^d \rightarrow \bbR^d$ such that $T$ is Borel-measurable and $T_{\#}\mu = \nu$.
By now it is well-understood that Lipschitz transport maps are a powerful tool to transfer analytic properties such as functional inequalities from one measure to another, see for instance \cite{mikulincer2023LipschitzPropertiesTransportation} for an overview of applications.

A natural candidate for this is the map $T$ coming from optimal transport (OT).
In a seminal work, Caffarelli \cite{caffarelli2000MonotonicityPropertiesOptimal} has proven that the OT map from the Gaussian measure to a uniformly log-concave measure is Lipschitz continuous. A generalization to OT maps from a uniformly $\log$-concave measure to compactly supported pertubation has been provided in \cite{colombo2015LipschitzChangesVariables},
see also \cite{kolesnikov2011MassTransportationContractions} for a review.
Due to the nature of optimal transport maps being the solution of an optimization problem, there is no general explicit construction known. This makes it quite challenging to obtain further (quantitative) regularity results. 

A more constructive approach towards existence of Lipschitz maps between measures has been proposed in the seminal work by Kim and Milman \cite{kim2012GeneralizationCaffarellisContraction}. Their construction of a transport map is based on the time-reversal of Langevin dynamics. We will therefore call this map Langevin transport map. Since then there have been a series of works proving Lipschitz continuity of the Langevin transport map under various assumptions, see for instance \cite{neeman2022LipschitzChangesVariables,klartag2023SpectralMonotonicityGaussian,fathi2024TransportationLogLipschitzPerturbations,lopez-rivera2024BakryEmeryApproachLipschitz,chewi2024Uniformin$N$LogSobolevInequality,stephanovitch2024SmoothTransportMap}. Similar techniques have also been employed along the Polchinski flow in \cite{serres2024BehaviorPoincareConstant,shenfeld2024ExactRenormalizationGroups}.

The recent works \cite{fathi2024TransportationLogLipschitzPerturbations} and \cite{lopez-rivera2024BakryEmeryApproachLipschitz} both provide a proof on the bound of the Lipschitz constant of the Langevin transport map between a uniformly log-concave probability measure and a $\log$-Lipschitz perturbation. The purpose of our work is to provide an alternative proof to bound the Lipschitz constant in this setting by means of a coupling argument. For this we leverage the fact that the vector fields inducing the Langevin transport map admit a probabilistic representation in form of a stochastic control problem to prove gradient and Hessian bounds using reflection coupling \cite{eberle2016ReflectionCouplingsContraction} adapted to control, already used in \cite{conforti2023CouplingReflectionControlled} and \cite{cecchin2024ExponentialTurnpikePhenomenon} to study ergodicity properties of (non-linear) optimal control problems. The flexibility of coupling arguments also allows to for treatment of different perturbations as already hinted in \cite{conforti2024WeakSemiconvexityEstimates}. We reserve a detailed study in this direction for future work.

Let us also remark that there are other approaches based on reflection coupling to prove analytic properties of asymptotic $\log$-concave measures such as a Poincaré inequality or a weak gradient commutation estimate \cite{cattiaux2014SemiLogConcaveMarkov}.

This note should be rather thought of as a proof of concept of the robustness of coupling techniques. As a cherry on top, we manage to relax the hypothesis of uniform $\log$-concavity to a mere asymptotic $\log$-concavity, defined precisely below, and to provide an alternative hypothesis to the control in the third derivative, which is needed in both \cite{fathi2024TransportationLogLipschitzPerturbations} and \cite{lopez-rivera2024BakryEmeryApproachLipschitz}.

The article is structured as follows. In \Cref{sec:transport_map_opt_cont} we provide a formal description of the Langevin transport map and how it is related to an optimal control problem. \Cref{sec:coup_refl} is devoted to the description of coupling by reflection applied to controlled dynamics and its contractive properties. With this at hand we can prove gradient and Hessian estimates in \Cref{sec:bds_Hess} which are needed to prove Lipschitzianity of the Langevin transport map.

\paragraph{Acknowledgements:} \emph{Giovanni Conforti wishes to dedicate this note to Patrick Cattiaux, thanking him for  his support and for contributing to shape his scientific interests.} \\Katharina Eichinger acknowledges the support of the Agence nationale de la recherche, through the PEPR PDE-AI project (ANR-23-PEIA-0004).

\section{Construction of the Langevin transport map and relation to optimal control}\label{sec:transport_map_opt_cont}
\subsection{Construction of the Langevin transport map by Kim and Milman}
We start by describing the (formal) construction of the Langevin transport map introduced by Kim and Milman \cite{kim2012GeneralizationCaffarellisContraction} by inverting the overdamped Langevin dynamics, where we refer to \cite{fathi2024TransportationLogLipschitzPerturbations, mikulincer2023LipschitzPropertiesTransportation} for the rigorous justification of the construction. Let $\mu(\De x) \propto e^{-U(x)}\De x$ 
for $U \in C^1(\bbR^d)$. Denote by $(X^x_t)_{t\geq 0}$ be the overdamped Langevin dynamics associated to $\mu$ starting at $x \in \bbR^d$, i.e.
\begin{equation}\label{eq:Langevin_SDE}
    \De X^x_t = -\nabla U(X^x_t) \De t + \sqrt{2} \De B_t, \quad X^x_0 = x.
\end{equation}
Denote by $(P_t)_{t\geq 0}$ the associated semigroup, defined by
\begin{equation}
    P_tf(x) := \bbE[f(X^x_t)].
\end{equation}

Suppose from now on that $x$ is a random variable distributed according to $\nu$ with $\nu(\De x) \propto e^{-W(x)}\mu(\De x)$ for $W:\bbR^d \rightarrow \bbR$.
It is well known that $\mu_t = \cL(X_t)$, more precisely its density w.r.t. the Lebesgue measure $\rho_t = \frac{\De \mu_t}{\De \cL^d}$, satisfies the continuity equation
\begin{equation}
    \partial_t \rho_t = \nabla \cdot (\rho_t(\nabla U + \nabla \log \rho_t)) = \nabla \cdot (\rho_t \nabla \log \frac{\De \mu_t}{\De \mu}), \quad \rho_0 = \frac{\De \nu}{\De \cL^d}.
\end{equation}
Now using the semigroup representation $\frac{\De \mu_t}{\De \mu} = P_t e^{-W}$ we arrive at
\begin{equation}
    \partial_t \rho_t =  \nabla \cdot (\rho_t \nabla \log P_t e^{-W}), \quad \rho_0 = \frac{\De \nu}{\De \cL^d}.
\end{equation}
Let us set $V_t = - \log P_t e^{-W}$. By virtue of the continuity equation the vector field $\nabla V_t$ induces a transport map from $\mu_{0}$ to $\mu_{t}$ via the map $S_t$ defined by
\begin{equation}
    \frac{\De}{\De t} S_t = \nabla V_t(S_t), \quad S_0 = \Id,
\end{equation}
i.e. $(S_t)_{\#}\mu_{0} = \mu_{t}$. 

Suppose now that $U$ is nice enough such that $(P_t)_{t\geq 0}$ is ergodic, implying that $(\mu_t)_{t\geq 0}$ converges weakly to $\mu$ as $t \rightarrow \infty$. 
Upon justifying the limit we hence get for $S := \lim_{t \rightarrow \infty} S_t$ that $S_{\#}\nu = \mu$. To get a transport map from $\mu$ to $\nu$, we invert the maps. Define $T_t := S_t^{-1}$, so $(T_t)_{\#}\mu_{t} = \mu_{0} = \nu$. Then again setting $T = \lim_{t \rightarrow \infty} T_t$ under suitable regularity hypothesis we have $T_{\#}\nu = \mu$. 

The map $T$ corresponds to the map constructed by Kim and Milman \cite{kim2012GeneralizationCaffarellisContraction}, which we call here Langevin transport map. The proof strategy we employ in this paper directly gives Lipschitz continuity on both $S$ and $T$, so we decided to include $S$ as well. The key lemma to obtain the Lipschitz estimates of these transport maps is the following lemma, see e.g. \cite[Lemma 11]{mikulincer2023LipschitzPropertiesTransportation}.
\begin{lemma}\label{lem:Est_Lip_heat_flow}
Assume that 
\begin{equation}\label{eq:bound_convexity}
    \lambda_{\max}(t) \succeq \nabla^2 V_t \succeq \lambda_{\min}(t),
\end{equation}
then 
\begin{itemize}
    \item $S$ is Lipschitz with constant $\exp(\int_0^\infty \lambda_{\max}(t)\De t)$,
    \item $T$ is Lipschitz with constant $\exp(-\int_0^\infty \lambda_{\min}(t)\De t)$.
\end{itemize}

\end{lemma}

With this in mind, let us state our assumptions to prove bounds on the Hessian of $V_t$. 
Instead of requiring uniform convexity of the potential $U: \bbR^d \rightarrow \bbR$ as in the works \cite{fathi2024TransportationLogLipschitzPerturbations, lopez-rivera2024BakryEmeryApproachLipschitz}, we relax this hypothesis to an asymptotic convexity assumption, formulated in terms of the weak convexity profile of $U$
\begin{equation}
\kappa_{U}(r) = \inf\left\{\frac{\ip{\nabla U(x)-\nabla U(\hat{x})}{x-\hat{x}}}{|x-\hat{x}|^2} : |x-\hat{x}|=r \right\}\eqsp,
\end{equation}
which has already been used in \cite{eberle2016ReflectionCouplingsContraction} to prove exponential ergodicity of \eqref{eq:Langevin_SDE}.

\begin{assumption}\label{ass:U_and_W_Lip}
Assume that $U \in C^2(\bbR^d)$ and $W \in C^1(\bbR^d)$ satisfy
\begin{enumerate}[label=(\roman*)]
\item The convexity profile of $U$ satisfies $\kappa_U \in C( (0,+\infty), \bbR)$ and
\begin{equation}
     \int_{0}^1 r (\kappa_U(r))^- \, \De r < +\infty \, , \quad \liminf_{r \to +\infty} \kappa_U(r) > 0,
\end{equation}
where $(\kappa_U(r))^- := \max\{-\kappa_U(r), 0  \}$.
\item $W$ is Lipschitz continuous with constant $\const^W_1$.
\end{enumerate}
\end{assumption}
\noindent We furthermore require some more regularity on $U$, either \Cref{ass:U_bd_2nd_der} or \Cref{ass:U_bd_3rd_der}.

\begin{assumption}\label{ass:U_bd_2nd_der}
    Assume that $U \in C^2(\bbR^d)$  and that there exists $\const^U_2 \in \bbR$ such that
\begin{equation}
    \sup_{u,x \in \bbR^d, |u|=1} |\ip{u}{\nabla^2 U(x)u}| = \const^U_2 < +\infty.
\end{equation}
\end{assumption}

\begin{assumptionp}{\ref*{ass:U_bd_2nd_der}$'$}\label{ass:U_bd_3rd_der}
    Assume that $U \in C^2(\bbR^d)$ and
    \begin{enumerate}[label=(\roman*)]
    \item\label{item:unif_convex}
    \begin{equation}
        \inf_{u,x \in \bbR^d, |u|=1} \ip{u}{\nabla^2 U(x)u} =: \alpha > - \infty 
    \end{equation}
    \item $\nabla^2 U$ is Lipschitz with constant $\const^U_3$
    \end{enumerate}
\end{assumptionp}

\begin{remark}
If $\alpha > 0$ \Cref{ass:U_bd_3rd_der} together with \Cref{ass:U_and_W_Lip} corresponds exactly to the assumptions imposed in \cite{fathi2024TransportationLogLipschitzPerturbations} in the Euclidean setting. But we will see in our proofs that uniform convexity of $U$ is not necessary for us. 

Up to our knowledge assumptions without any condition on the third derivative for a Lipschitz map towards $e^{-U-W}\De x$ comparable to \Cref{ass:U_bd_2nd_der} are only handled in \cite{brigati2024HeatFlowLogconcavity} where they prove Lipschitzianity of the Langevin transport map from the Gaussian distribution to $e^{-U-W}\De x$. 
We note that the potential associated to the Gaussian distribution automatically satisfies both \Cref{ass:U_bd_2nd_der} and \Cref{ass:U_bd_3rd_der}. Similar results could also be obtained using coupling by following the approach in \cite{conforti2024WeakSemiconvexityEstimates} but we decided to focus on one method for this note.

Here we prove Lipschitzianity of the map from $e^{-U} \De x$ to $e^{-U-W}\De x$, which is different than in \cite{brigati2024HeatFlowLogconcavity}.
\end{remark}
\noindent The following theorem is our main result.

\begin{theorem}\label{thm:heat_flow_Lip_Lip_pert}
Assume \Cref{ass:U_and_W_Lip}, then we have the following properties of the Langevin transport map $T$ and its inverse $S$.
Set
\begin{equation}
    \bar\kappa(r) = \kappa_{U}(r) - \frac{4 \const^W_1}{ C_{\kappa_U}r},
\end{equation}
where $0 < C_{\kappa_U} \leq 1$ is defined in \eqref{def:lambda_kappa_C_kappa}. Furthermore, denote by $0 < \lambda_{\kappa_U}, \lambda_{\bar\kappa} < + \infty$, $0\leq C_{\bar\kappa} \leq 1$ the constants in \eqref{def:lambda_kappa_C_kappa} for the convexity profiles $\kappa_U$ and $\bar\kappa$.
\begin{enumerate}[label=(\roman*)]
\item\label{item:Lip_Hess_bd} If in addition \Cref{ass:U_bd_2nd_der} holds, then both $T$ and $S$ are Lipschitz continuous with constant
\begin{equation}
\exp\left(\frac{2\const^W_1}{\sqrt{\pi\lambda_{\bar{\kappa}}}C_{\bar{\kappa}}} \left(\frac{3e^{\frac{\lambda_{\kappa_U}}{2\lambda_{\bar\kappa}}}\const^U_2}{2\lambda_{\kappa_U}C_{\kappa_U}}
 + \sqrt{2}  \right) \right).
\end{equation}
\item\label{item:Lip_3rd_der_bd} If in addition \Cref{ass:U_bd_3rd_der} holds, then both $T$ and $S$ are Lipschitz continuous with constant
\begin{equation}
\exp\left(\frac{\const^W_1}{C_{\bar{\kappa}}}\left(\frac{\const^U_3}{\lambda_{\kappa_U}\lambda_{\bar\kappa}C_{\kappa_U}} + \frac{3|\alpha| e^{\frac{\lambda_{\kappa_U}}{2\lambda_{\bar\kappa}}}}{2\sqrt{\pi\lambda_{\bar\kappa}}\lambda_{\kappa_U}C_{\kappa_U}} 
   +  \frac{2\sqrt{2}}{\sqrt{\pi\lambda_{\bar\kappa}}} \right)  \right).
\end{equation}
If in addition $\alpha >0$, this bound improves to
\begin{equation}\label{eq:our_bound_vs_Max}
\exp\left( \frac{\const^W_1}{C_{\bar{\kappa}}} \left(  \frac{\const^U_3}{2\lambda_{\bar{\kappa}}\alpha} + \frac{\sqrt{2}}{\sqrt{\pi\lambda_{\bar{\kappa}}}} + \frac{\sqrt{2\lambda_{\bar{\kappa}}}}{\sqrt{\pi e}(\alpha + \lambda_{\bar{\kappa}})} \right) \right).
\end{equation}
\end{enumerate}
\end{theorem}

\begin{remark}
The last bound on the Lipschitz constant \eqref{eq:our_bound_vs_Max} is established under the same hypothesis as in the Euclidean setting in \cite{fathi2024TransportationLogLipschitzPerturbations}. However, our bound is in general larger compared to their bound
\begin{equation}
    \exp\left(10 \const^W_1\left(\frac{1}{\sqrt{\alpha}}+ \frac{\const^W_1}{\alpha} + \frac{\const^U_3}{\alpha^2} \right) \right).
\end{equation}
This can be seen by computing explicitly from \eqref{def:lambda_kappa_C_kappa}

\begin{equation}
    C_{\bar\kappa} = \frac{1}{2} \exp\left(-2\frac{(\const^W_1)^2}{\alpha} \right),
\end{equation}
yielding a double exponential dependence on $\const^W_1$ in our bound. Together with $\lambda_{\bar\kappa} < \alpha$ one observes that our bound cannot be better in general.

By inspecting our proof we observe that this comes from the fact that $\bar\kappa$ is a crude lower bound on the convexity profile of the potential fields $(V_t)_{t\geq0}$ giving rise to the Langevin transport map. One could improve the bound $\bar\kappa$ by using our Hessian estimates together with an early stopping argument from optimal control; however this is of a high technical cost, and we prefer leave it as it is for readability of the main argument.
\end{remark}

\subsection{Control interpretation}\label{sec:contr_interpretation}

In order to prove \Cref{thm:heat_flow_Lip_Lip_pert}, we will use the fact that $V_t$ is related to a stochastic optimal control problem in the following sense.

For a fixed time horizon $T>0$ set $\varphi_t = V_{T-t} = -\log P_{T-t}e^{-W}$. Then differentiating $t \mapsto \varphi_t$ we see that $\varphi$ solves the Hamilton-Jacobi-Bellman (HJB) equation
\begin{equation}\label{eq:HJB_log_backward_density}
\partial_t \varphi_t + \Delta \varphi_t -\nabla U\cdot \nabla \varphi_t -|\nabla\varphi_t|^2=0, \quad \varphi_T=W.
\end{equation}
It is well known that under mild regularity hypothesis (which are in particular implied by \Cref{ass:U_and_W_Lip}-\Cref{ass:U_bd_2nd_der} or \Cref{ass:U_and_W_Lip}-\Cref{ass:U_bd_3rd_der}, we make this rigorous in the proposition below), there exists a unique classical solution to \cref{eq:HJB_log_backward_density} and it is given by the value function of the following control problem
\begin{equation}\label{eq:stoch_cont_prob}
\begin{split}
\varphi_t(x)=\inf_{u} \bbE&[\int_t^T\frac14 |u_s|^2 d s+ W(X^{x,u}_T)]\\
&\begin{cases}
\De X^{x,u}_s = \big(-\nabla U(X^{x,u}_s) +u_s\big)ds + \sqrt{2}\De B_s, \\
X^{x,u}_0=x.
\end{cases} 
\end{split}
\end{equation}
where an admissible control $u_s$ is a square integrable process adapted to the filtration generated by $(B_s)_{s\in[0,T]}$. 

For convenience, we summarize all the properties which we need on the control problem in the following proposition.

\begin{proposition}\label{prop:Control_properties}
Assume that \Cref{ass:U_and_W_Lip} holds.
\begin{enumerate}[label=(\roman*)]
\item\label{item:fin_dim_HJB} The value function  $(t,x)\mapsto \varphi_{t}(x)$ defined by \eqref{eq:stoch_cont_prob} is the unique classical solution in the set $C([0,T]\times \bbR^d) \cap C^{1,2}_{\mathrm{loc}}((0,T) \times \bbR^d)$ of the HJB equation \eqref{eq:HJB_log_backward_density}.
\item\label{item:fin_dim_OptContr} The map 
$
(s,x) \mapsto -2\nabla\varphi_s(x)\eqsp, 
$
is an optimal Markov control policy in the sense that if for any given $(t,x) \in\bbR_+\times \bbR^{d}$ we define the process $(X_s)_{s\in[t,T]}$ as the unique strong solution of the SDE
\begin{equation}\label{eq:fin_dim_OptDyn}
    \De X_s =[-\nabla U(X_s) -2\nabla\varphi_s(X_{s})] \De s  + \sqrt{2}\De B_s \eqsp,\quad X_t=x
\end{equation}
and we set
$
    u_{s}:= -2\nabla\varphi_s(X_{s})\eqsp, 
$
then $u$ is an optimal control, i.e. a minimizer in \eqref{eq:stoch_cont_prob}
and the process $(X_s)_{s\in[t,T]}$ coincides with $(X^{u}_s)_{s\in[t,T]}$ a.s..
\item\label{item:Pontryagin} If in addition \Cref{ass:U_bd_2nd_der} or \Cref{ass:U_bd_3rd_der} holds, then  setting $Y^x_s=\nabla\varphi_s(X^x_s)$, $Z^x_s=\nabla^2\varphi_s(X^x_s)$ we obtain that $(X^x_\cdot,Y^x_\cdot,Z^x_\cdot)$ form a solution of the Pontryagin system
\begin{equation}\label{eq:Pontryagin_sys}
\begin{split}
\De X^x_s &=(-\nabla U(X^x_s)  -2 Y^x_s )\De s+ \sqrt{2}\De B_s, \quad X^x_0=x\\
\De Y^x_s &= \nabla^2 U(X^x_s) Y^x_s\De s + \sqrt{2}Z^x_s  \De B_s, \quad Y^x_T=\nabla W(X^x_T)
\end{split}
\end{equation}

\end{enumerate}
    
\end{proposition}

For a proof under the current setting we refer to \cite[Proposition 4.1, 4.5]{cecchin2024ExponentialTurnpikePhenomenon}.

\section{Coupling by reflection}\label{sec:coup_refl}
The purpose of this section to give an overview of the reflection coupling introduced in \cite{lindvall1986CouplingMultidimensionalDiffusions} whose contractive properties have been shown in \cite{eberle2016ReflectionCouplingsContraction}, applied to controlled processes, as in \cite{conforti2023CouplingReflectionControlled, cecchin2024ExponentialTurnpikePhenomenon}.

Given a family of potentials $(V_s)_{s\geq0}$ and a bounded measurable family of  drift fields $(\alpha_s)_{s\geq0}$, $\alpha_s: \bbR^d \rightarrow \bbR^d$, we seek to study the ergodic behaviour of the SDE
\begin{equation}\label{eq:SDE_prototype}
    \De X_s = \left(-\nabla V_s(X_s) + \alpha_s(X_s)\right) \De s + \sqrt{2} \De B_s
\end{equation}
by means of coupling by reflection.
\begin{remark}
The whole construction and ergodic properties reported here do not require the drift coming from a potential and work for a general drift fields $\beta_s$ satisfying the appropriate conditions, as described in \cite{cecchin2024ExponentialTurnpikePhenomenon}. We decided to describe everything for SDEs coming from potentials since here we are interested on their ergodic behaviour coming from the weak convexity profile defined below. For the proofs we refer to \cite{cecchin2024ExponentialTurnpikePhenomenon}.
\end{remark}
To this end, we define the weak convexity profile of $V_s$ and $\cV :=(V_s)_{s\geq0}$ as
\begin{equation}\label{eq:def_fun_kappa_inhom}
\kappa_{V_s}(r) = \inf\left\{\frac{\ip{\nabla V_s(x)-\nabla V_s(\hat{x})}{x-\hat{x}}}{|x-\hat{x}|^2} : |x-\hat{x}|=r \right\}\eqsp, \quad 
\kappa_{\cV}(r) =\inf \{ \kappa_{V_s}(r): s \in \bbR^+  \}, \quad r>0 \eqsp.
\end{equation}
We assume the following on the drift.
\begin{assumption}\label{ass:coupling}
\begin{enumerate}[wide, labelwidth=!, labelindent=0pt,label=(\roman*)]
\item  For every $s\geq 0$ the vector field $x \mapsto -\nabla V_s(x)$ is continuous and uniformly one-sided Lipschitz continuous, i.e. there exists $\rho \in \bbR$ such that for all $x, \hat{x} \in \bbR^d$
\begin{equation}
    \ip{\nabla V_s(x)-\nabla V_s(\hat{x})}{x-\hat{x}} \geq \rho |x-\hat{x}|^2.
\end{equation}

In addition, for any compact set $\msc \subset \bbR^d$,  $t \mapsto \sup_{x \in\msc} |\nabla V_t(x)| \leq \upphi_{\msc}(t)$ is locally integrable.
\item\label{ass:_kappa} We suppose that there exists $\bar{\kappa} \in \msk$ such that for all $r> 0$ $\kappa_{\cV}(r) \geq \bar{\kappa}(r)$, where
\begin{equation}
  \label{eq:def_msk}
\msk = \{ \kappa \in C( (0,+\infty), \bbR) \,: \,  \int_{0}^1 r (\kappa(r))^- \, \De r < +\infty \, , \, \liminf_{r \to +\infty} \kappa(r) > 0\} \eqsp,
\end{equation}
\end{enumerate}
\end{assumption}

In particular under \Cref{ass:coupling}, there exists unique strong
solutions for \eqref{eq:SDE_prototype} by
\cite[Corollary 2.6]{gyongy1996ExistenceStrongSolutions}. 
In the sequel, the contractive properties of the couplings we shall consider will be expressed either in the total variation distance or in some modification of the $1$-Wasserstein distance $W_1$, which we borrow from \cite{eberle2016ReflectionCouplingsContraction}.

For $\kappa\in \msk$ consider
\begin{equation}\label{eq:aux_fun_eberle_2}
   R_0=\inf\{R\geq 0: \inf_{r \geq R}\kappa(r)\geq 0 \} \eqsp, \quad  R_1=\inf\defEns{ R\geq R_0: \inf_{r \geq R} \{ \kappa(r)R(R-R_0) \} \geq 8 } \eqsp,
  \end{equation}
and define
\begin{equation}\label{eq:def_Eberle_func}
\begin{aligned}
&\phi_{\kappa}(r)=\exp\Big(-\frac{1}{4}\int_0^r s(\kappa(s))^-\De s\Big) , \quad
 \Phi_{\kappa}(r)=\int_0^r\phi_{\kappa}(s)\De s, \\
 &g_{\kappa}(r)=1- \frac{\int_0^{r \wedge R_1}\Phi_{\kappa}(s)/\phi_{\kappa}(s) \De s}{2\mathcal{Z}_{\kappa}} \eqsp,
\end{aligned}
\end{equation}
where $(\kappa(s))^- = \max\{-\kappa(s),0\}$, and $\mathcal{Z}_{\kappa} = \int_0^{R_1}\Phi_{\kappa}(s)/\phi_{\kappa}(s) \De s$.
With this notation at hand we are ready to construct the afore-mentioned modifications of the Wasserstein distance to achieve contraction properties.
\begin{definition}\label{def:twisted_metric}
For $\kappa \in \msk$ define $f_\kappa : \bbR_+ \to \bbR_+$  as
\begin{equation}
  \label{eq:def_f_kappa}
f_{\kappa}(r) =\int_0^r\phi_{\kappa}(s)g_{\kappa}(s)\De s \eqsp,  \quad r \geq 0  \eqsp.
\end{equation}
For $\kappa \in \msk$, we define for any $\mu,\hat{\mu}\in\cP_1(\bbR^d)$
\begin{equation}\label{eq:def_W_f}
W_{f_\kappa}(\mu,\hat{\mu}) = \inf_{\pi\in\Pi(\mu,\hat{\mu})} \int f_\kappa(|x-\hat{x}|) \pi(\De x \De \hat{x}),
\end{equation}
where $\Pi(\mu,\hat\mu)$ is the set of couplings of $\mu$ and $\hat\mu$
\begin{equation}
\Pi(\mu,\hat\mu)=\{ \pi\in\cP(\bbR^{2d}): \pi(A\times\bbR^d)=\mu(A),\,\,\pi(\bbR^d\times A)=\hat\mu(A) \quad \forall A\subseteq \mathcal{B}(\bbR^d) \}.
\end{equation}
\end{definition}
 
 \begin{proposition}\label{prop:Lyapunov_funct_coup_by_ref}
 For $\kappa\in \msk$ we define  $\lambda_{\kappa}>0$ and $C_{\kappa} \geq 0$ as
\begin{equation}
  \label{def:lambda_kappa_C_kappa}
  \lambda_{\kappa}=\frac{2}{\int_0^{R_1}\Phi_{\kappa}(s)/\phi_{\kappa}(s) \De s} \eqsp, \quad
   C_{\kappa}= \frac{\phi_{\kappa}(R_0)}{2} \eqsp,
\end{equation}
where $\Phi_{\kappa}$ and $\phi_{\kappa}$ are defined in \eqref{eq:def_Eberle_func}.
Then the following holds
\begin{enumerate}[label=(\roman*)]
\item\label{item_1:Lyapunov_funct_coup_by_ref} $f_{\kappa}$ is concave, continuously differentiable such that $f'_{\kappa}$ is absolutely continuous. Furthermore, it is equivalent to the identity $r \mapsto r$ on $\bbR_+$:
\begin{equation}\label{eq:Lyapunov_funct_coup_by_ref_2}
  C_{\kappa}\, r\leq f_{\kappa}(r) \leq r\eqsp, \quad C_{\kappa}\leq f'_{\kappa}(r) \leq 1\eqsp,
\end{equation}
\item\label{item_2:Lyapunov_funct_coup_by_ref}  The differential inequality
\begin{equation}\label{eq:Eberle_funct_2}
4 f''_{\kappa}(r)-r\kappa(r)f'_{\kappa}(r)\leq -\lambda_{\kappa} f_{\kappa}(r) \eqsp.
\end{equation}
holds for all $r>0$.
\item\label{item_3:Lyapunov_funct_coup_by_ref}  $\lambda_{\kappa},C_{\kappa}$ are monotone in the following sense: if $\kappa,\kappa'\in K$ are such that 
 \begin{equation}
\kappa(r)\geq\kappa'(r) \quad \forall r>0, \qquad \mbox{ then } \quad 
 \lambda_{\kappa}\geq\lambda_{\kappa'}, \quad C_{\kappa}\geq C_{\kappa'}.
 \end{equation}
 \end{enumerate}
 \end{proposition}

With this at hand let us recall the construction of controlled coupling by reflection introduced in \cite{conforti2023CouplingReflectionControlled} and extended in \cite{cecchin2024ExponentialTurnpikePhenomenon} and its contraction properties.
For an $\cF_0$-measurable random variable $\zeta=(\zeta_1,\zeta_2) \in \bbR^{2d}$, consider for $t < T_0 = \inf\{s \geq 0 \,: \, X_s \neq \hat{X}_s \}$,
\begin{equation}\label{eq:coup_by_ref_contraction_2}
  \begin{aligned}
\De X_t& = \{ - \nabla V_t(X_t) + {\alpha}_t(X_t) \}\De t+\sqrt{2} \De B^1_t , \eqsp X_0 = \zeta_1 \eqsp, \\
\De \hat{X}_t& =\{ - \nabla V_t(\hat{X}_t) +{\alpha}_t(X_t) \}\De t+ \sqrt{2} \De \hat{B}^1_1 , \eqsp \hat{X}_0 = \zeta_2 \eqsp,
\end{aligned}
\end{equation}
where
\begin{equation}
  \label{eq:def_mirror_brownian}
 \De\hat{B}^1_t= (\Id-2 \,e_t \cdot e^{\transpose}_t)\cdot \De B^1_t \eqsp, \quad e_t = (X_t-\hat{X}_t)/|X_t-\hat{X}_t| \eqsp,
\end{equation}
and for $t \geq T_0$,
\begin{equation}\label{eq:coup_by_ref_contraction_2_2}
  \De X_t = - \nabla V_t(X_t) + \alpha_t(X_t)\De t+\sqrt{2}\De B^1_t  \eqsp, \quad \hat{X}_t = X_t \eqsp.
\end{equation} 
\begin{proposition}
  Assume \Cref{ass:coupling} and that
  $\alpha$ is bounded and locally Lipschitz continuous. 
  Then the system \eqref{eq:coup_by_ref_contraction_2}-\eqref{eq:coup_by_ref_contraction_2_2} admits a strong solution. 
\end{proposition}

\begin{remark}
By Lévy's characterization, $(\hat{B}^1_s)_{s \geq}$ is a Brownian motion. 
Therefore, if $\alpha\equiv 0$, $(X_s,\hat{X}_s)_{s\geq 0}$ is indeed a coupling of the diffusion process \eqref{eq:SDE_prototype} with initial laws given by $(\zeta_1, \zeta_2)$.
\end{remark}

Note that if $\alpha\equiv0$, \eqref{eq:coup_by_ref_contraction_2}-\eqref{eq:coup_by_ref_contraction_2_2} recovers usual coupling by reflection for non-homogeneous drifts. When $\alpha\not \equiv 0$ this is not a coupling in the classical sense because, for example, the drift of $(\hat{X}_t)_{t \geq 0}$ may depend on $(X_t)_{t \geq 0}$ and therefore $(\hat{X}_t)_{t \geq 0}$ may not even be a Markov process. However, this construction will be key to obtain uniform in time Lipschitz estimates for the value function of stochastic control problems. In this framework, $(X_t)_{t \geq 0}$ and $(u_t=\alpha_t(X_t))_{t \geq 0}$ represent  an optimal process and an optimal control for a given stochastic control problem. On the other hand, $(\hat{X}_t)_{t \geq 0}$ is an admissible process relative to the suboptimal control $(u_t)_{t \geq 0}$ for a stochastic control problem with different initial conditions.

 Let us turn now to the contractive properties we can prove using this controlled coupling by reflection.
\begin{proposition}\label{prop:contr_same_drift}
Assume \Cref{ass:coupling} and that
  $\alpha$ is bounded and locally Lipschitz continuous. 
  Suppose that $\zeta$ has a finite first order moment. Let
$(X_s,\hat{X}_s)_{s\geq 0}$ be a solution of  \eqref{eq:coup_by_ref_contraction_2}-\eqref{eq:coup_by_ref_contraction_2_2} and denote by $\mu_t$ and $\hat{\mu}_t$ the distribution of $X_t$ and $\hat{X}_t$ respectively for $t \geq 0$.
Let $\bar{\kappa} \in \msk$ such that for any $r>0$ $\kappa_{\cV}(r) \geq \bar{\kappa}(r)$. The following holds.
\begin{enumerate}[label=(\roman*)]
\item\label{item_2:contraction_coup_by_ref}
For any $t \geq 0$,
\begin{equation}
\bbE[f_{\bar{\kappa}}(|X_t-\hat{X}_t|)]\leq e^{-\lambda_{\bar{\kappa}} t} \bbE[ f_{\bar{\kappa}}(|X_0-\hat{X}_0|)] \eqsp.
\end{equation}
 As a corollary, we have 
 $$W_{f_{\bar{\kappa}}}(\mu_t,\hat{\mu}_t)\leq e^{-\lambda_{\bar{\kappa}} t} W_{f_{\bar{\kappa}}}(\mu_0,\hat{\mu}_0).$$
\item\label{item:final_TV_coup_by_ref}
For any $t \geq 0$,
    \begin{equation}
        \bbP(X_t\neq\hat{X}_t)\leq q^{\bar{\kappa}}_t  \bbE[ f_{\bar{\kappa}}(|X_0-\hat{X}_0|)]    \label{eq:1} \eqsp,
      \end{equation}
 where for $\kappa\in \msk$, $t\geq 0$ we define
\begin{equation}\label{q_kappa_t_def}
q^{{\kappa}}_t =
\begin{cases}
\frac{1}{\sqrt{2\pi t}C_{{\kappa}}} & t < \frac{1}{2\lambda_{\kappa}}\\
    \frac{\sqrt{\lambda_{\kappa} e}}{{\sqrt{\pi}C_{{\kappa}}}} e^{-\lambda_{\kappa} t} &  t \geq  \frac{1}{2\lambda_{\kappa}}
\end{cases}.
\end{equation} 
As a corollary, we have 
$$\|\mu_t-\hat\mu_t \|_{\TV}\leq q^{\bar\kappa}_t W_{f_{\bar{\kappa}}}(\mu_0,\hat{\mu}_0).$$
\end{enumerate}
  \end{proposition}

\section{Bounds on the Hessian}\label{sec:bds_Hess}
In order to prove \Cref{thm:heat_flow_Lip_Lip_pert}, we use the control interpretation of $\varphi_{t}=-\log P_{T-t}e^{-W}$ highlighted in \Cref{sec:contr_interpretation} and the (controlled) coupling by reflection described in \Cref{sec:coup_refl} to prove suitable bounds on the Hessian of $V_t = \varphi_{T-t}$.

Let us start by proving a gradient estimate.
\begin{proposition}\label{prop:Lip_pert_grad_bound}
Under \Cref{ass:U_and_W_Lip} the following gradient estimate holds true
\begin{equation}
    \|\nabla \varphi_t \|_\infty \leq C_{\kappa_U}^{-1} e^{-\lambda_{\kappa_U} (T-t)}\const^W_1.
\end{equation}
If in addition \Cref{ass:U_bd_3rd_der}-\ref{item:unif_convex} is in place with $\alpha > 0$, then this bound improves to
\begin{equation}
    \|\nabla \varphi_t \|_\infty \leq e^{-\alpha (T-t)}\const^W_1.
\end{equation}
\end{proposition}
\begin{proof}
Let $x, \hat{x} \in \bbR^d$. Denote by $(X_s)_{s \in [t,T]}$ the unique strong solution to the optimal dynamics for for the control problem \eqref{eq:stoch_cont_prob}, which is given by \Cref{prop:Control_properties}
\begin{equation}
\De X_s =[-\nabla U(X_s) -2\nabla\varphi_s(X_{s})] \De s  + \sqrt{2}\De B_s \eqsp,\quad X_t=x.
\end{equation}
Furthermore, define the coupled process $(\hat{X}_s)_{s \in [t,T]}$ given by the controlled coupling by reflection \eqref{eq:coup_by_ref_contraction_2}-\eqref{eq:coup_by_ref_contraction_2_2} as
\begin{equation}
\De \hat{X}_s =[-\nabla U(\hat{X}_s) -2\nabla\varphi_s(X_{s})] \De s  + \sqrt{2}\De \hat{B}_s \eqsp,\quad X_t=x.
\end{equation}
Since $(-2\nabla\varphi_s(X_{s}))_{s \in [t,T]}$ is an optimal control associated to the initial condition $x$ and Brownian motion $(B_s)_{s \in[t,T]}$ and an admissible control for initial condition $\hat{x}$ and Brownian motion $(\hat{B}_s)_{s \in[t,T]}$ we obtain
\begin{align}
\varphi_t(x) &= \bbE[\int_t^T |\nabla \varphi(X_s)|^2 d s+ W(X_T)], \\
\varphi_t(\hat{x}) &\leq \bbE[\int_t^T |\nabla \varphi(X_s)|^2 d s+ W(\hat{X}_T)].
\end{align}
This gives, using Lipschitz continuity of $W$, \Cref{prop:Lyapunov_funct_coup_by_ref}-\ref{item_1:Lyapunov_funct_coup_by_ref} and finally the contraction estimate given in \Cref{prop:contr_same_drift}-\ref{item_2:contraction_coup_by_ref}
\begin{align}
\varphi_t(\hat{x}) - \varphi_t(x) &\leq \bbE[W(\hat{X}_T) - W(X_T)] \\
&\leq \frac{\const^W_1}{C_{\kappa_U}} \bbE[f_{\kappa_U}(|\hat{X}_T - X_T|)] \\
&\leq \frac{\const^W_1}{C_{\kappa_U}}  e^{\lambda_{\kappa_U}(T-t)}f_{\kappa_U}(|x - \hat{x}|)
\end{align}
We conclude the first estimate since $f_{\kappa_U}(|x - \hat{x}|) \leq |x - \hat{x}|$ again by \Cref{prop:Lyapunov_funct_coup_by_ref}-\ref{item_1:Lyapunov_funct_coup_by_ref} and exchanging the roles of $x$ and $\hat{x}$.

The second estimate readily follows by observing that if $\alpha >0$, i.e. $U$ uniformly convex, then we can choose $f_{\kappa_U}=\mathrm{Id}$, $C_{\kappa_U}=1$ and $\lambda_{\kappa_U}=\alpha$.
\end{proof}
With this gradient estimate we have a sufficient control on the monotonicity profile of the optimally controlled dynamics for the control problem \eqref{eq:stoch_cont_prob}, whose drift is given by $-\nabla U - 2\nabla \varphi_t$, i.e. coming from the time-dependent potential $\psi_t := U + 2 \varphi_t$. 
Indeed, we have thanks to \Cref{prop:Lip_pert_grad_bound}
\begin{equation}
    \kappa_{\psi_t}(r) \geq \kappa_{U}(r) - \frac{4 \const^W_1}{ C_{\kappa_U}r} e^{-\lambda_{\kappa_U} (T-t)} \geq \bar{\kappa}(r) :=  \kappa_{U}(r) - \frac{4 \const^W_1}{ C_{\kappa_U}r}.
\end{equation}
From this we obtain the following contraction estimate for optimally controlled dynamics by applying \Cref{prop:contr_same_drift}.
\begin{corollary}\label{cor:contr_opt_dyn}
Let \Cref{ass:U_and_W_Lip} hold. Setting
\begin{equation}\label{eq:def_bar_kappa}
    \bar{\kappa}(r) =  \kappa_{U}(r) - \frac{4 \const^W_1}{ C_{\kappa_U}r},
\end{equation}
we have $\bar{\kappa} \in \msk$ and for a coupling by reflection $(X_s, \hat{X}_s)_{s \in [t,T]}$ of the optimal dynamics \eqref{eq:fin_dim_OptDyn} with initial condition $(x,\hat{x})$
\begin{equation}
    \bbE[f_{\bar\kappa}(|X_s-\hat{X}_s|)] \leq e^{-\lambda_{\bar{\kappa}}(s-t)} f_{\bar{\kappa}}(|x-\hat{x}|),
\end{equation}
and
\begin{equation}
    \bbP(X_s\neq\hat{X}_s) \leq q^{\bar\kappa}_{s-t} f_{\bar{\kappa}}(|x-\hat{x}|),
\end{equation}
where $q^{\bar\kappa}_{s-t}$ has been defined in \eqref{q_kappa_t_def}.
\end{corollary}
With this at hand we can now to use the Pontryagin system \eqref{eq:Pontryagin_sys} to obtain Hessian estimates on $\varphi$ to conclude Lipschitz continuity of the Langevin transport map and its inverse with \Cref{lem:Est_Lip_heat_flow}. Here, for a matrix valued function $A: \bbR^d \rightarrow \bbR^{d \times d}$ we define
\begin{equation}
    \| A \|_{\infty} := \sup\left\{|\ip{u}{A(x)u}|: x,u \in \bbR^d, |u|=1 \right\}.
\end{equation}
\begin{theorem}\label{thm:Hess_est}
We have the following Hessian bounds on the value function of the control problem \eqref{eq:stoch_cont_prob}.
\begin{enumerate}[label=(\roman*)]
\item\label{item:Hess_est_bd_2nd}  If \Cref{ass:U_and_W_Lip}-\Cref{ass:U_bd_2nd_der} holds, then
\begin{align}
\|\nabla^2 \varphi_t \|_{\infty}
&\leq 2\const^W_1 \left(\frac{\const^U_2}{C_{\kappa_U}}\left(\frac{e^{\frac{\lambda_{\kappa_U}}{2\lambda_{\bar\kappa}}}}{2\sqrt{\pi\lambda_{\bar\kappa}}C_{\bar\kappa}} e^{-\lambda_{\kappa_U}(T-t)}
+ \frac{\sqrt{\lambda_{\bar\kappa}}e^{\frac{\lambda_{\kappa_U}}{2\lambda_{\bar\kappa}}}}{(\lambda_{\kappa_U}-\lambda_{\bar\kappa})\sqrt{\pi}C_{\bar\kappa}}\left[e^{-\lambda_{\bar\kappa}(T-t)} - e^{-\lambda_{\kappa_U}(T-t)} \right]\right) + q^{\bar{\kappa}}_{T-t}   \right).
\end{align}
\item\label{item:Hess_est_bd_3rd} If \Cref{ass:U_and_W_Lip}-\Cref{ass:U_bd_3rd_der} holds, then
\begin{align}
\|\nabla^2 \varphi_t \|_{\infty} 
\leq &\, \frac{\const^W_1}{C_{\kappa_U}}\left[\frac{\const^U_3\left(e^{-\lambda_{\bar\kappa}(T-t)} - e^{-\lambda_{\kappa_U}(T-t)}\right)}{C_{\bar{\kappa}}(\lambda_{\kappa_U}-\lambda_{\bar\kappa})} + 2|\alpha| e^{\frac{\lambda_{\kappa_U}}{2\lambda_{\bar\kappa}}}\left(\frac{e^{-\lambda_{\kappa_U}(T-t)}}{2\sqrt{\pi\lambda_{\bar\kappa}}C_{\bar\kappa}} 
+ \frac{\sqrt{\lambda_{\bar\kappa}}\left[e^{-\lambda_{\bar\kappa}(T-t)} - e^{-\lambda_{\kappa_U}(T-t)} \right]}{(\lambda_{\kappa_U}-\lambda_{\bar\kappa})\sqrt{\pi}C_{\bar\kappa}} \right)  \right] \\
&+ 2 \const^W_1q^{\bar{\kappa}}_{T-t}
\end{align}
If moreover, $\alpha > 0$, then this estimate improves to
\begin{equation}
\|\nabla^2 \varphi_t \|_{\infty} \leq \const^W_1 \left(\frac{C_{\bar{\kappa}}^{-1}\const^U_3}{2 \alpha - \lambda_{\bar{\kappa}}}\left(e^{-\lambda_{\bar{\kappa}}(T-t)} - e^{-2\alpha(T-t)} \right) + 2 e^{-\alpha(T-t)} q^{\bar{\kappa}}_{T-t}   \right).
\end{equation}
\end{enumerate}

\end{theorem}
\begin{proof}
Let $x,\hat{x} \in \bbR^d$. Consider coupling by reflection \eqref{eq:coup_by_ref_contraction_2}-\eqref{eq:coup_by_ref_contraction_2_2} of the optimally controlled dynamics of \eqref{eq:stoch_cont_prob}
\begin{align}
dX_s &=(-\nabla U(X_s)  - 2 Y_s )\De s+ \sqrt{2}\De B_s, \quad X_0=x \\
d\hat{X}_s &=(-\nabla U(\hat{X}_s)  - 2 \hat{Y}_s )\De s+ \sqrt{2}\De \hat B_s, \quad X_0=\hat{x}
\end{align}
with $Y_s = \nabla \varphi_s(X_s)$, $\hat{Y}_s = \nabla \varphi_s(\hat{X}_s)$ satisfying by \Cref{prop:Control_properties}-\ref{item:Pontryagin}
\begin{align}
\De Y_s &= \nabla^2 U(X_s) Y_s\De s + \sqrt{2}Z_s \De B_s, \quad Y_T=\nabla W(X_T), \\
\De \hat{Y}_s &= \nabla^2 U(\hat{X}_s) \hat{Y}_s\De s + \sqrt{2}\hat{Z}_s \De \hat B_s, \quad \hat{Y}_T=\nabla W(\hat{X}_T).
\end{align}
Then applying Itô's formula to the convex function $\sqrt{|Y_s - \hat{Y_s}|^2 + a}$ for $a>0$, dropping the non-negative quadratic variation term and letting $a \rightarrow 0$, we arrive at
\begin{align}
\De |Y_s - \hat{Y}_s| &\geq \frac{1}{|Y_s - \hat{Y}_s|}\ip{Y_s - \hat{Y}_s}{\nabla^2 U(X_s) Y_s -  \nabla^2 U(\hat{X}_s) \hat{Y}_s} \De s + \De M_s,
\end{align}
where $M_s$ stands for a martingale term.
Under assumption \Cref{ass:U_bd_2nd_der} we continue as follows
\begin{align}
\De |Y_s - \hat{Y}_s| 
&\geq -\const^U_2(|Y_s| + |\hat{Y}_s| ) \IND_{\{X_s\neq \hat{X}_s\}} + \De M_s \\
&\geq -2\frac{\const^U_2\const^W_1}{C_{\kappa_U}} e^{-\lambda_{\kappa_U} (T-s)} \IND_{\{X_s\neq \hat{X}_s\}} + \De M_s,
\end{align}
where we have used \Cref{prop:Lip_pert_grad_bound} for the second inequality. 
Taking expectation, this gives with \Cref{cor:contr_opt_dyn}
\begin{align}
\bbE[|Y_t - \hat{Y}_t|] &\leq \bbE[|\nabla W(X_T) - \nabla W(\hat{X}_T)|] + 2\frac{\const^U_2\const^W_1}{C_{\kappa_U}} \int_t^T e^{-\lambda_{\kappa_U} (T-s)} \bbP(X_s \neq \hat{X}_s) \De s \\
&\leq 2 \const^W_1q^{\bar{\kappa}}_{T-t} f_{\bar\kappa}(|x-\hat{x}|) + 2\frac{\const^U_2\const^W_1}{C_{\kappa_U}} \int_t^T e^{-\lambda_{\kappa_U} (T-s)} q^{\bar{\kappa}}_{s-t} \De s f_{\bar\kappa}(|x-\hat{x}|).
\end{align}
Since $Y_t = \nabla\varphi_t(x)$, $\hat{Y}_t = \nabla\varphi_t(\hat{x})$ and $f_{\bar\kappa}(|x-\hat{x}|) \leq |x-\hat{x}|$ we conclude with the explicit computations postponed to \Cref{lem:integrals}.

If instead \Cref{ass:U_bd_3rd_der} is in place, we write
\begin{align}
\De |Y_s - \hat{Y}_s|
&\geq \frac{1}{|Y_s - \hat{Y}_s|} \left( \langle Y_s - \hat{Y}_s, \nabla^2 U(\hat{X}_s) (Y_s - \hat{Y}_s)\rangle 
 + \langle Y_s - \hat{Y}_s,(\nabla^2 U(X_s) - \nabla^2 U(\hat{X}_s)) Y_s \rangle \right) + \De M_s \\
 &\geq \left(\inf_{x,u \in \bbR^d, |u|=1} \ip{u}{\nabla^2 U(x)u}\right) |Y_s - \hat{Y}_s| - |\nabla^2 U(X_s) - \nabla^2 U(\hat{X}_s)||Y_s| \\
 &\geq \alpha  (|Y_s| + |\hat{Y}_s|)\IND_{\{X_s\neq \hat{X}_s\}} - \frac{\const^U_3}{C_{\bar{\kappa}}} |Y_s| f_{\bar\kappa}(|X_s - \hat{X}_s|) \De s + \De M_s,\label{est:lowerbound_Hess_U}
\end{align}
where in the last inequality we have used the assumption and \Cref{prop:Lyapunov_funct_coup_by_ref}.
Taking expectation and using \Cref{prop:Lip_pert_grad_bound}  yields
\begin{align}
\bbE[|Y_t - \hat{Y}_t|] 
&\leq \bbE[|\nabla W(X_T) - \nabla W(\hat{X}_T)|] + \int_t^T \bbE\left[\frac{\const^U_3}{C_{\bar{\kappa}}} |Y_s|f_{\bar\kappa}(|X_s - \hat{X}_s|)- \alpha  (|Y_s| + |\hat{Y}_s|)\IND_{\{X_s\neq \hat{X}_s\}} \right] \De s \\
&\leq 2\const^W_1 \bbP(X_T \neq \hat{X}_T) + \frac{\const^W_1}{C_{\kappa_U}} \int_t^T e^{-\lambda_{\kappa_U} (T-s)}\left[\frac{\const^U_3}{C_{\bar{\kappa}}} \bbE[f_{\bar\kappa}(|X_s - \hat{X}_s|)] +2|\alpha|  \bbP(X_s\neq\hat{X}_s)\right] \De s.
\end{align}
Finally, using the contraction estimates from \Cref{cor:contr_opt_dyn} gives
\begin{align}
\bbE[|Y_t - \hat{Y}_t|] 
&\leq 2 \const^W_1q^{\bar{\kappa}}_{T-t} f_{\bar\kappa}(|x-\hat{x}|) + \frac{\const^W_1}{C_{\kappa_U}} \int_t^T e^{-\lambda_{\kappa_U} (T-s)}\left[\frac{\const^U_3}{C_{\bar{\kappa}}}e^{-\lambda_{\bar\kappa}(s-t)} + 2|\alpha| q^{\bar\kappa}_{s-t} \right]\De s f_{\bar\kappa}(|x-\hat{x}|),
\end{align}
and we conclude again with the computations postponed to \Cref{lem:integrals}.

If in addition $\alpha >0$ we can instead apply Grönwall's lemma to \eqref{est:lowerbound_Hess_U}, take expectation and use \Cref{prop:Lip_pert_grad_bound} to obtain
\begin{align}
\bbE[|Y_t - \hat{Y}_t|] 
&\leq \frac{\const^U_3}{C_{\bar{\kappa}}} \int_t^T e^{-\alpha(T-s)}\bbE[f_{\bar\kappa}(|X_s - \hat{X}_s|) |Y_s|] \De s + e^{-\alpha(T-t)}\bbE[|\nabla W(X_T) - \nabla W(\hat{X}_T)|] \\
&\leq \frac{\const^U_3\const^W_1}{C_{\bar{\kappa}}} \int_t^T e^{-2\alpha(T-s)}\bbE[f_{\bar\kappa}(|X_s - \hat{X}_s|)] \De s + 2\const^W_1 e^{-\alpha(T-t)}\bbP(X_T \neq \hat{X}_T).
\end{align}
Now applying \Cref{cor:contr_opt_dyn} gives
\begin{align}
\bbE[|Y_t - \hat{Y}_t|] 
&\leq \frac{\const^U_3 \const^W_1}{C_{\bar{\kappa}}} \left( \int_t^T e^{-2\alpha(T-s)}e^{-\lambda_{\bar{\kappa}}(s-t)} \De s \right) f_{\bar\kappa}(|x-\hat{x}|) + 2\const^W_1 e^{-\alpha(T-t)} q^{\bar{\kappa}}_{T-t}  f_{\bar\kappa}(|x-\hat{x}|) \\
&= \const^W_1 f_{\bar\kappa}(|x-\hat{x}|) \left(\frac{C_{\bar{\kappa}}^{-1}\const^U_3}{2 \alpha - \lambda_{\bar{\kappa}}}\left(e^{-\lambda_{\bar{\kappa}}(T-t)} - e^{-2\alpha(T-t)} \right) + 2e^{-\alpha(T-t)} q^{\bar{\kappa}}_{T-t}   \right),
\end{align}
and we conclude as before.
\end{proof}
We are now in position to prove the main theorem.
\begin{proof}[Proof of \Cref{thm:heat_flow_Lip_Lip_pert}]
We recall that $V_t = \varphi_{T-t}$. This allows us to use the Hessian estimates obtained in \Cref{thm:Hess_est}.
\begin{itemize}
\item \underline{Proof of \ref{item:Lip_Hess_bd}}\\
\noindent If \Cref{ass:U_bd_2nd_der} is in place, we obtain the following Hessian bound on $V_t$ from \Cref{thm:Hess_est}-\ref{item:Hess_est_bd_2nd}
\begin{align}
\|\nabla^2 V_t \|_{\infty}
&\leq 2\const^W_1 \left(\frac{\const^U_2}{C_{\kappa_U}}\left(\frac{e^{\frac{\lambda_{\kappa_U}}{2\lambda_{\bar\kappa}}}}{2\sqrt{\pi\lambda_{\bar\kappa}}C_{\bar\kappa}} e^{-\lambda_{\kappa_U}t}
+ \frac{\sqrt{\lambda_{\bar\kappa}}e^{\frac{\lambda_{\kappa_U}}{2\lambda_{\bar\kappa}}}}{(\lambda_{\kappa_U}-\lambda_{\bar\kappa})\sqrt{\pi}C_{\bar\kappa}}\left[e^{-\lambda_{\bar\kappa}t} - e^{-\lambda_{\kappa_U}t} \right]\right) + q^{\bar{\kappa}}_{t}   \right).
\end{align}
Integrating this bound with the help of the computations of \Cref{lem:integrals} yields the result by applying  \Cref{lem:Est_Lip_heat_flow}.
\item \underline{Proof of \ref{item:Lip_3rd_der_bd}}\\
Under \Cref{ass:U_bd_2nd_der} we obtain with \Cref{thm:Hess_est}-\ref{item:Hess_est_bd_3rd}

\begin{align}
\|\nabla^2 V_t \|_{\infty} 
\leq &\, \frac{\const^W_1}{C_{\kappa_U}}\left[\frac{\const^U_3\left(e^{-\lambda_{\bar\kappa}t} - e^{-\lambda_{\kappa_U}t}\right)}{C_{\bar{\kappa}}(\lambda_{\kappa_U}-\lambda_{\bar\kappa})} + 2|\alpha| e^{\frac{\lambda_{\kappa_U}}{2\lambda_{\bar\kappa}}}\left(\frac{e^{-\lambda_{\kappa_U}t}}{2\sqrt{\pi\lambda_{\bar\kappa}}C_{\bar\kappa}} 
+ \frac{\sqrt{\lambda_{\bar\kappa}}\left[e^{-\lambda_{\bar\kappa}t} - e^{-\lambda_{\kappa_U}t} \right]}{(\lambda_{\kappa_U}-\lambda_{\bar\kappa})\sqrt{\pi}C_{\bar\kappa}} \right)  \right] \\
&+ 2 \const^W_1q^{\bar{\kappa}}_{t}.
\end{align}
With the help of \Cref{lem:integrals} and  \Cref{lem:Est_Lip_heat_flow} we obtain the result.
If in addition $\alpha>0$ we get
\begin{equation}\label{eq:bd_Hessian_VF}
    \|\nabla^2 V_t\|_{\infty} \leq \const^W_1 \left(\frac{C_{\bar{\kappa}}^{-1}\const^U_3}{2 \alpha - \lambda_{\bar{\kappa}}}\left(e^{-\lambda_{\bar{\kappa}}t} - e^{-2\alpha t} \right) + 2e^{-\alpha t} q^{\bar{\kappa}}_{t}   \right)
\end{equation}
Now we conclude with \Cref{lem:Est_Lip_heat_flow} using \Cref{lem:integrals}
\begin{align}
&\int_0^\infty \const^W_1 \left(\frac{C_{\bar{\kappa}}^{-1}\const^U_3}{2 \alpha - \lambda_{\bar{\kappa}}}\left(e^{-\lambda_{\bar{\kappa}}t} - 2e^{-2\alpha t} \right) + 2e^{-\alpha t} q^{\bar{\kappa}}_{t}   \right) \De t  \\
\leq &\, \const^W_1 \left( \frac{C_{\bar{\kappa}}^{-1}\const^U_3}{2\lambda_{\bar{\kappa}}\alpha} + \frac{\sqrt{2}}{\sqrt{\pi\lambda_{\bar{\kappa}}}C_{\bar{\kappa}}} + \frac{\sqrt{2\lambda_{\bar{\kappa}}}}{\sqrt{\pi e}C_{\bar{\kappa}}(\alpha + \lambda_{\bar{\kappa}})} \right) \\
= &\, \const^W_1 C_{\bar{\kappa}}^{-1} \left(  \frac{\const^U_3}{2\lambda_{\bar{\kappa}}\alpha} + \frac{\sqrt{2}}{\sqrt{\pi\lambda_{\bar{\kappa}}}} + \frac{\sqrt{2\lambda_{\bar{\kappa}}}}{\sqrt{\pi e}(\alpha + \lambda_{\bar{\kappa}})} \right).
\end{align}
\end{itemize}
\end{proof}

\appendix
\section{Technical calculations}
Here we collect some technical calculations.
\begin{lemma}\label{lem:integrals}
In the setting of \Cref{thm:heat_flow_Lip_Lip_pert}, resp. \Cref{thm:Hess_est} we have the following estimates.
\begin{itemize}
\item 
\begin{equation}
\int_t^T e^{-\lambda_{\kappa_U} (T-s)} q^{\bar{\kappa}}_{s-t} \De s 
\leq \frac{e^{\frac{\lambda_{\kappa_U}}{2\lambda_{\bar\kappa}}}}{2\sqrt{\pi\lambda_{\bar\kappa}}C_{\bar\kappa}} e^{-\lambda_{\kappa_U}(T-t)}
+ \frac{\sqrt{\lambda_{\bar\kappa}}e^{\frac{\lambda_{\kappa_U}}{2\lambda_{\bar\kappa}}}}{(\lambda_{\kappa_U}-\lambda_{\bar\kappa})\sqrt{\pi}C_{\bar\kappa}}\left[e^{-\lambda_{\bar\kappa}(T-t)} - e^{-\lambda_{\kappa_U}(T-t)} \right]
\end{equation}
\item 
\begin{equation}
    \int_0^\infty e^{-\alpha t} q^{\bar{\kappa}}_{t} \De t 
    \leq  \frac{1}{\sqrt{2\pi\lambda_{\bar{\kappa}}}C_{\bar{\kappa}}} + \frac{\sqrt{\lambda_{\bar{\kappa}}}}{\sqrt{2\pi e}C_{\bar{\kappa}}(\alpha + \lambda_{\bar{\kappa}})} 
\end{equation}
\item
\begin{equation}
    \int_0^\infty q^{\bar\kappa}_t \De t = \frac{\sqrt{2}}{\sqrt{\pi\lambda_{\bar\kappa}}C_{\bar\kappa}}
\end{equation}
\end{itemize}
\end{lemma}
\begin{proof}
For the first integral we compute
\begin{align}
\int_t^T e^{-\lambda_{\kappa_U} (T-s)} q^{\bar{\kappa}}_{s-t} \De s 
&= \frac{1}{\sqrt{2\pi}C_{\bar\kappa}}\int_t^{t+\frac{1}{2\lambda_{\bar\kappa}}} e^{-\lambda_{\kappa_U}(T-s)}\frac{1}{\sqrt{s-t}} \De s
+ \frac{\sqrt{\lambda_{\bar\kappa}e}}{\sqrt{\pi}C_{\bar\kappa}}\int_{t+\frac{1}{2\lambda_{\bar\kappa}}}^T e^{-\lambda_{\kappa_U}(T-t)}e^{-\lambda_{\bar\kappa}(s-t)} \De s\\
&\leq \frac{e^{\frac{\lambda_{\kappa_U}}{2\lambda_{\bar\kappa}}}}{2\sqrt{\pi\lambda_{\bar\kappa}}C_{\bar\kappa}} e^{-\lambda_{\kappa_U}(T-t)}
+ \frac{\sqrt{\lambda_{\bar\kappa}e}}{(\lambda_{\kappa_U}-\lambda_{\bar\kappa})\sqrt{\pi}C_{\bar\kappa}}\left[e^{-\lambda_{\bar\kappa}(T-t)} - e^{-\lambda_{\kappa_U}(T-t - \frac{1}{2\lambda_{\bar\kappa}}) - \frac{1}{2}} \right].
\end{align}
The bound on the second integral follows similarly
\begin{align}
\int_0^\infty e^{-\alpha t} q^{\bar{\kappa}}_{t} \De t 
&= \frac{1}{2\sqrt{\pi}C_{\bar{\kappa}}}\int_0^{\frac{1}{2\lambda_{\bar{\kappa}}}} e^{-\alpha t} \frac{1}{\sqrt{t}} \De t + \frac{\sqrt{\lambda_{\bar{\kappa}} e}}{\sqrt{2\pi}C_{\bar{\kappa}}}\int_{\frac{1}{2\lambda_{\bar{\kappa}}}}^\infty e^{-(\alpha + \lambda_{\bar{\kappa}})t} \De t \\
&= \frac{1}{\sqrt{2\pi\alpha}C_{\bar{\kappa}}} \int_0^{\sqrt{\frac{\alpha}{\lambda_{\bar{\kappa}}}}} e^{-\frac{s^2}{2}} \De s + \frac{\sqrt{\lambda_{\bar{\kappa}} e}}{\sqrt{2\pi}C_{\bar{\kappa}}} \frac{e^{-(\alpha+\lambda_{\bar{\kappa}})/(2\lambda_{\bar{\kappa}})}}{\alpha + \lambda_{\bar{\kappa}}} \\
&\leq  \frac{1}{\sqrt{2\pi\lambda_{\bar{\kappa}}}C_{\bar{\kappa}}} + \frac{\sqrt{\lambda_{\bar{\kappa}}}}{\sqrt{2\pi e}C_{\bar{\kappa}}(\alpha + \lambda_{\bar{\kappa}})}.
\end{align}
The value of the last integral follows by noting that
\begin{align}
\int_0^T q^{\bar{\kappa}}_{s} \De s 
= \frac{1}{\sqrt{2\pi \lambda_{\bar{\kappa}}}C_{\bar{\kappa}}} \left(1 + \sqrt{e}(e^{-1/2} - e^{-\lambda_{\bar{\kappa}}T }) \right) = \frac{1}{\sqrt{2\pi\lambda_{\bar{\kappa}}} C_{\bar{\kappa}}}(2 - \sqrt{e} e^{-\lambda_{\bar{\kappa}}T }) ,
\end{align}
and taking $T \rightarrow \infty$.
\end{proof}

\printbibliography
\end{document}